\documentclass[12pt,reqno]{amsart}

\usepackage{amsmath,amsthm,amssymb,comment,fullpage}
\usepackage{braket}
\usepackage{mathtools}
\usepackage{flexisym}
\usepackage{caption}
\usepackage{breqn}
\usepackage{amssymb}
\usepackage{float}
\usepackage{mathrsfs}
\DeclareMathAlphabet{\mathzc}{OT1}{pzc}{m}{it}
\usepackage{latexsym}
\usepackage{amssymb}
\usepackage{array}

\usepackage{caption}
\usepackage{times}
\usepackage[T1]{fontenc}
\usepackage{mathrsfs}
\usepackage{svg}
\usepackage{latexsym}
\usepackage{epsfig}
\usepackage{amsmath,amsfonts,amsthm,amssymb,amscd}
\input amssym.def
\input amssym.tex
\usepackage{color}
\usepackage{hyperref}
\usepackage{url}
\newcommand{\bburl}[1]{\textcolor{blue}{\url{#1}}}

\usepackage{tikz}
\usepackage{tkz-tab}
\usetikzlibrary{shapes.geometric,positioning}

\usepackage{scalerel}

\def\vecsign#1{\rule[1.388\LMex]{\dimexpr#1-2.5pt}{.36\LMpt}%
  \kern-6.0\LMpt\mathchar"017E}

\usepackage{stackengine,amsmath}
\stackMath
\usepackage{graphicx}

\newcommand{\burl}[1]{\textcolor{blue}{\url{#1}}}

\newcommand{\abs}[1]{\left|#1\right|}

\numberwithin{equation}{section}

\newtheorem{thm}{Theorem}[section]

\newtheorem{defi}[thm]{Definition}

\newtheorem{cla}[thm]{Claim}

\theoremstyle{plain}

\newtheorem{lemma}[thm]{Lemma}

\newtheorem{theorem}[thm]{Theorem}

\newcommand\be{\begin{equation}}
\newcommand\ee{\end{equation}}
\newcommand\bee{\begin{equation*}}
\newcommand\eee{\end{equation*}}
\newcommand\bea{\begin{eqnarray}}
\newcommand\eea{\end{eqnarray}}
\newcommand\beae{\begin{eqnarray*}}
\newcommand\eeae{\end{eqnarray*}}
\newcommand\bi{\begin{itemize}}
\newcommand\ei{\end{itemize}}
\newcommand\ben{\begin{enumerate}}
\newcommand\een{\end{enumerate}}
\newcommand\bc{\begin{center}}
\newcommand\ec{\end{center}}
\newcommand\ba{\begin{array}}
\newcommand\ea{\end{array}}

\newcommand{\Z}{\ensuremath{\mathbb{Z}}}
\newcommand{\Q}{\mathbb{Q}}

\newcommand\frakfamily{\usefont{U}{yfrak}{m}{n}}
\DeclareTextFontCommand{\textfrak}{\frakfamily}

\newcommand\leg[2]{{#1\overwithdelims () #2}}

\newcommand{\hr}[1]{\href{#1}{\url{#1}}}

\newsavebox\EllipticCurve
\savebox{\EllipticCurve}(200,180)[cc]{%
\setlength{\unitlength}{.6pt}
\begin{picture}(200,180)
\thicklines
\put(0,90){\begin{picture}(200,90)
   \qbezier(18,0)(18,22.5)(54,22.5)
   \qbezier(54,22.5)(69.75,20.25)(85.5,18)
   \qbezier(85.5,18)(101.25,13.5)(117,13.5)
   \qbezier(117,13.5)(153,13.5) (189,85.5)
\end{picture}
} \put(0,90){\begin{picture}(200,90)
   \qbezier(18,0)(18,-22.5)(54,-22.5)
   \qbezier(54,-22.5)(69.75,-20.25)(85.5,-18)
   \qbezier(85.5,-18)(101.25,-13.5)(117,-13.5)
   \qbezier(117,-13.5)(153,-13.5) (189,-85.5)
\end{picture}
}
\end{picture}
}

\title{Class Numbers and Pell's Equation $x^2 + 105y^2 = z^2$}

\author{Thomas Jaklitsch}
\email{\textcolor{blue}{\href{mailto:thomasjaklitsch@college.harvard.edu}{thomasjaklitsch@college.harvard.edu}}}
\address{Department of Mathematics, Harvard University, Cambridge, MA 02138}

\author{Thomas C. Martinez}
\email{\textcolor{blue}{\href{mailto:tmartinez@hmc.edu}{tmartinez@hmc.edu}}}
\address{Department of Mathematics, Harvey Mudd College, Claremont, CA 90701}

\author{Steven J. Miller}
\email{\textcolor{blue}{\href{mailto:sjm1@williams.edu}{sjm1@williams.edu}},  \textcolor{blue}{\href{Steven.Miller.MC.96@aya.yale.edu}{Steven.Miller.MC.96@aya.yale.edu}}}
\address{Department of Mathematics and Statistics, Williams College, Williamstown, MA 01267}

\author{Sagnik Mukherjee}
\email{\textcolor{blue}{\href{mailto:smukherjee@cmi.ac.in}{smukherjee@cmi.ac.in}}}
\address{Department of Mathematics, Chennai Mathematical Institute, Siruseri, Chennai, Tamilnadu-603103}

\thanks{We thank Amnon Yekutieli for introducing us to the problem and sharing his work on the subject, and our colleagues from the 2021 Polymath REU, Leart Ajvazaj, Manyi Guo, Dylan Jamner, Yuan Lu, Jonathan Marvel-Zuccola, Sydney Morgan, and Bangqi (Blair) Yuan, for numerous helpful conversations and comments, especially Sydney and Blair, who worked with the third named author on a preliminary research project which was the springboard for this work. Those lectures, notes and questions were a valuable starting point for the present paper.}



\date{\today}

\begin{document}

\maketitle

\begin{abstract} Two well-studied Diophantine equations are those of Pythagorean triples and elliptic curves, for the first we have a
parametrization through rational points on the unit circle, and for the second we have a structure theorem for the group of rational solutions.
Recently Yekutieli discussed a connection between these two problems, and described the group structure of Pythagorean triples and the
number of triples for a given hypotenuse. In \cite{JMMM} we generalized these methods and results to Pell's equation. We find a similar group structure and count
on the number of solutions for a given $z$ to $x^2 + Dy^2 = z^2$ when $D$ is 1 or 2 modulo 4 and the class group of $\mathbb{Q}[\sqrt{-D}]$ is a free
$\Z_2$ module, which always happens if the class number is at most 2. 
In this paper we discuss the main results of \cite{JMMM} using some concrete examples in the case of $D=105$.
\end{abstract}

\tableofcontents

\section{Introduction}
\subsection{Background} 
A Diophantine equation is a polynomial of finite degree with integer coefficients, where we often are interested in the integer or rational solutions. These problems have fascinated mathematicians for millenia. Perhaps the most famous equation and solutions are the Pythagorean triples, triples of integers $(a, b, c)$ that satisfy \be f(x,y,z) \ = \ x^2 + y^2 - z^2 \ = \ 0.\ee  These triples can be represented in multiple forms, and a particularly useful one is to convert to points on the unit circle: $(a/c, b/c)$.


In studying Diophantine equations, there are several natural questions.

\begin{itemize}

\item Are there any solutions? Answering this is often non-trivial, as the analysis of Fermat's Last Theorem showed: If $x, y, z$ are positive integers then there are no solutions to $x^n + y^n = z^n$ if $n$ is an integer at least three, see \cite{TW, Wi}. Related problems, such as the Beal conjecture ($x^a + y^b = z^c$ with certain constraints) are still open. \\ \

\item If there are solutions, are there finitely many or infinitely many? For example, Mihailescu \cite{Met, Mih1, Mih2} proved Catalan's Conjecture: the only two consecutive positive powers are 8 and 9, explicitly, if $x, y > 0$ and $a, b > 1$ then the only integer solution to $x^a - y^b = 1$ is $3^2 - 2^3 = 1$.\\ \

\item If there are infinitely many solutions, is it possible to parameterize them, and if so is there a structure on the set of solutions?

\end{itemize}

The answers to these questions are well-known for Pythagorean triples. We can generate all solutions by utilizing Euclid's formula: $a=m^2-n^2$, $b = 2mn$, $c=m^2+n^2$. A common way to find this parameterization is shown in Figure \ref{Circle}. There we start with one known solution to the corresponding problem on the unit circle: $(-1)^2 + 0^2 = 1$. Given any rational pair $(x, y)$ on the unit circle, the slope of the line from $(-1,0)$ to $(x,y)$ is rational, and conversely given any line with rational slope emanating from $(-1,0)$ it intersects the unit circle at a rational point, see for example \cite{MT-B} for details.

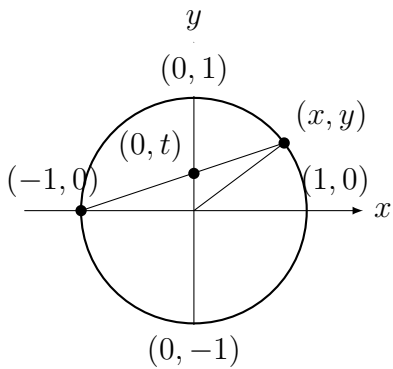
\begin{figure}[h]
\begin{center}
\begin{tikzpicture}[scale=1.5,cap=round,>=latex,baseline={(0,0)}]
    \draw[->] (-1.5cm,0cm) -- (1.5cm,0cm) node[right,fill=white] {$x$};
    \draw[->] (0cm,-1.5cm) -- (0cm,1.5cm) node[above,fill=white] {$y$};
    \draw[thick] (0cm,0cm) circle(1cm);
        \draw (-1.25cm,0cm) node[above=1pt] {$(-1,0)$}
    (1.25cm,0cm)  node[above=1pt] {$(1,0)$}
    (0cm,-1.25cm) node[fill=white] {$(0,-1)$}
    (0cm,1.25cm)  node[fill=white] {$(0,1)$};
\draw    (-1,-0) -- (0.8,0.6) ;
\draw    (0,-0) -- (0.8,0.6) ;
\fill[black] (-1,0) circle (0.5mm);
\fill[black] (0.8,0.6) circle (0.5mm)
node [above right] {$(x,y)$};
\fill[black] (0,0.33) circle (0.5mm)
node [above left] {$(0, t)$};
\end{tikzpicture}
\caption{A rational parametrization of the circle $x^2+y^2=1$.}
\label{Circle}
\end{center}
\end{figure}

What can we say about the \emph{structure} of the solutions to the Pythagorean equation? It is trivial to see that given a solution $(a, b, c)$ we can multiply each component by any integer $k > 1$ to obtain a new solution, we thus call triples that are relatively prime \emph{primitive} solutions. Are there other ways to generate new solutions from existing ones?

The answer is yes, and the motivation comes from an important generalization: elliptic curves (see \cite{Kn, Si, ST}). An elliptic curve in standard (or Weierstrass) form is the set of points $(x, y)$ satisfying the cubic equation
\begin{eqnarray}
    y^2\ =\ x^3 + ax + b,
\end{eqnarray} where $a,b\in\mathbb{Q}$ and the discriminant $4a^3+27b^2$ is not zero (we need the discriminant to be non-zero to avoid degenerate cases such as lines and parabolas). Mordell proved that the set of rational points forms a group, and the points of infinite order are isomorphic to $r$ copies of $\Z$.

\begin{theorem}[Mordell's theorem]
Let $E(\Q)$ be the set of rational points on an elliptic curve. Then $E(\Q)$ is a finitely generated abelian group of the form $\Z^r \oplus \mathbb{T}$, where $r$ is the geometric rank and $\mathbb{T}$ is a finite set of points of finite order.\footnote{Mazur \cite{Maz1, Maz2} proved that there are only 15 possibilities for $\mathbb{T}$: $\Z/N\Z$ for $N \in \{1,\dots, 10, 12\}$ and $\Z_2 \oplus \Z/2N\Z$ with $N \in \{1, \dots, 4\}$. The possible values of the rank are still a mystery. In 1938, Billing found an elliptic curve with rank $3$. The largest known rank increased slowly over the years, with the current record due to Elkies in 2006, and is rank at least $28$ (see \cite{Du} for a more comprehensive historical data on elliptic curve records). While originally it was conjectured that the rank can be arbitrarily large, now some models (see \cite{PPVW}) suggest that there may only be finitely many curves with rank exceeding 28.}
\end{theorem}

\begin{figure}
\begin{center}
\setlength{\unitlength}{.6pt}
\begin{picture}(200,180)
\put(-150,0){
\begin{picture}(200,180)
\put(-69.5,-59){\usebox{\EllipticCurve}} \thinlines
\put(10,95){\line(5,1){180}} \put(17,97){\circle*6}
\put(0,97){$P$} \put(79.5,109){\circle*6} \put(79.5,118){$Q$}
\put(151.5,122.5){\circle*6} \put(138,130){$R$}
\put(151.5,160){\qbezier[65](0,0)(0,-72.5)(0,-145)}
\put(151.5,59.5){\circle*6} \put(148,45){\makebox(0,0)[br]{$P\oplus
Q$}} \put(70,60){\makebox(0,0)[t]{\large$E$}}
\put(0,0){\makebox(220,-30)[c]{\small Addition of distinct points
$P$ and $Q$}}
\end{picture}
}
\put(120,0){
\begin{picture}(200,180)
\put(-69.5,-59.5){\usebox{\EllipticCurve}}
\put(10,104.5){\line(4,1){180}} \put(35,111){\circle*6}
\put(35,115){\makebox(0,0)[br]{$P$}} \put(170,145){\circle*6}
\put(165,145){\makebox(0,0)[br]{$R$}} \put(170,36){\circle*6}
\put(165,26){\makebox(0,0)[br]{$2P$}}
\put(170,165){\qbezier[65](0,0)(0,-72.5)(0,-145)}
\put(70,60){\makebox(0,0)[t]{\large$E$}}
\put(0,0){\makebox(220,-30)[c]{\small Adding a point $P$ to itself}}
\end{picture}
}
\end{picture}
\end{center}
\caption{\label{fig:ellcurve} The addition law on an elliptic curve.
In the second example the line is tangent to $E$ at $P$. Image courtesy of J. Silverman.}
\end{figure}

We demonstrate the group law on the set of rational solutions of an elliptic curve in Figure \ref{fig:ellcurve}. Thus given two points with rational coordinates, or even taking the same point twice, we can generate additional rational solutions.

Building on this success for elliptic curves, we revisit Pythagorean triples and see there is a similar structure. An excellent account of this was done recently by Yekutieli \cite{Ye}, who sets out a framework amenable to generalization by presenting Pythagorean triples as complex numbers. Thus operations on complex numbers (multiplication, complex conjugation) provide ways of creating new, different solutions from existing ones.



Explicitly, given a Pythaogrean triple $(a, b, c)$ we associate the complex number \be \zeta_{a,b,c} \ =\ x + iy\ =\ \frac{a}{c} + \frac{b}{c}i.\ee  There are four solutions where either $a$ or $b$ is zero: $1, i, -1, -i$. These are the units of $\Z[i] = \{a + ib: a, b \in \Z\}$, and correspond to trivial Pythagorean triples. We now consider $\zeta$ where both $a$ and $b$ are non-zero. We cannot have $a=b$, as that would lead to $\sqrt{2}$ being rational.\footnote{We would have $2a^2 = c^2$ and thus $\sqrt{2} = c/a \in \Q$.} A straightforward calculation shows that given such a solution $\zeta$ there are 7 other distinct conjugate solutions, we can multiply $\zeta$ by $i, i^2$ and $i^3$ (the units of $\Z[i]$ other than 1) and then we can take the complex conjugates of these four solutions. We illustrate this in Figure \ref{fig:complexpythagoras}, note without loss of generality given any Pythagorean triple not associated to a unit of $\Z[i]$ we may always adjust it, through multiplication by a unit and complex conjugation if needed, so that it lies in the shaded region (i.e., the second octant, or the part of the first quadrant where the imaginary part exceeds the real part).

\begin{figure}
\begin{center}
\scalebox{1}{\includegraphics{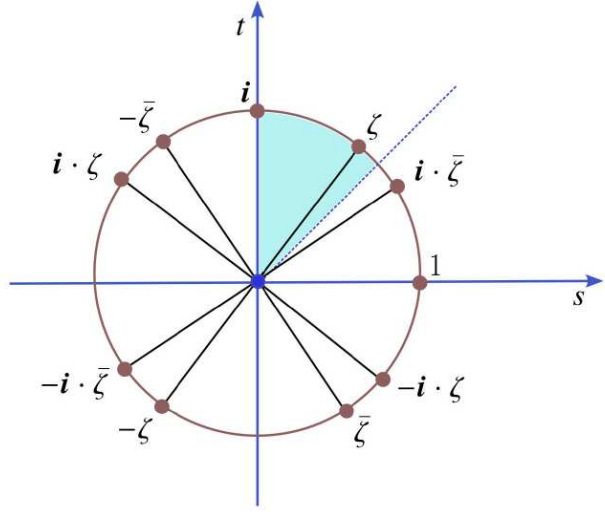}}
\caption{\label{fig:complexpythagoras} The four trivial solutions ($1, i, -1, -i$) and the seven conjugates to a non-trivial solution $\zeta$, which can be taken to lie in the second octant. Image from \cite{Ye}.}
\end{center}\end{figure}

Identifying Pythagorean triples with complex numbers yields a commutative group (if $z$ and $w$ are two complex numbers on the unit circle, the multiplicativity of the norm immediately implies $zw$ is also of norm 1 and hence on the unit circle). While rescaling a Pythagorean triple by $k$ does not change the complex number associated to it, multiplying the associated complex numbers generates new solutions. For example, the triple $(3,4,5)$ yields $\zeta_{3,4,5} = 3/5 + i 4/5$, and \be \zeta_{3,4,5}^2 \ = \ \left(\frac34 + \frac45 i\right) \left(\frac34 + \frac45 i\right) \ = \ -\frac{7}{25} + \frac{24}{25}i,\ee  which corresponds to the triple $(7, 24, 25)$, while \be  \zeta_{3,4,5} \ \ \zeta_{5,12,13} \ = \ \left(\frac35 + \frac45 i\right) \left(\frac{5}{13} + \frac{12}{13} i\right) \ = \ -\frac{33}{65} + \frac{56}{65} i, \ee  which corresponds to the triple $(33, 56, 65)$.


\subsection{Results of Yekutieli}

Yekutieli \cite{Ye} proved several results about the structure of the group of rational solutions to the unit circle version of the Pythagorean equation. 
We are interested in   the structure of positive integral solutions to the Diophantine equation
\begin{equation}\label{eq:pythagore}
    x^2+y^2 \ = \ z^2.
\end{equation}
To understand the main objective of \cite{Ye} first we need to introduce the \textit{normalized} solutions.

\begin{defi}\label{defi: normalized solutions for D=1}
A solution $(x_0,y_0,z_0)$ to \eqref{eq:pythagore} is defined to be normalized if $\gcd(x_0,y_0,z_0)=1$ and $x_0\le y_0$.
\end{defi}
 
Now, the main objective of \cite{Ye} is to find the number of normalized solutions\footnote{Note that given any solution $(x_0,y_0,z_0)$ to \eqref{eq:pythagore} we can generate infinitely many solutions out of this particular solutions, for example $(dx_0,dy_0,dz_0)$, $(ey_0,ex_0,ez_0)$ for any arbitrary integer $d$ and $e$. Thus by introducing this notion of normalized solution we are just considering the non-trivial generators of the solution space of \eqref{eq:pythagore}.} to \eqref{eq:pythagore} of the form $(a,b,c)$ for a given positive integer $c$.

In order to find the above, first he observed that every normalized solution $(a,b,c)$ to \eqref{eq:pythagore} can be encoded as the rational coordinate point $\zeta:= \frac{a}{c}+i\frac{b}{c}$ on the 2nd octant of the unit circle and vice-versa. Thus rational coordinate points on the 2nd octant of the unit circle is bijective to the set of normalized solution. Also, one can observe that \textit{up to sign}  the eight points (see \ref{fig:complexpythagoras})
$$\pm\zeta,\ \pm i\zeta, \ \pm \Bar{\zeta},\ \pm i\Bar{\zeta}$$
also correspond to the normalized solution $(a,b,c)$. Thus Yekutieli established a surjection, defined as \textbf{pt} from the set of rational coordinate points on unit circle onto the set of normalized solutions to \eqref{eq:pythagore}, by which he sends a rational coordinate point on the unit circle to the corresponding normalized solution, for example each of the above 8 points are mapped to $(a,b,c)$ via \textbf{pt}. So it is sufficient to focus on the rational coordinate points on the unit circle. 

Next, Yekutieli considered the rational coordinate points on the unit circle algebraically as $G(\Q)$, where
$$G(\Q) \ := \ \{a+ib:a,b\in \Q \text{ and } a^2+b^2=1\}.$$
Note that $G(\Q)$ is a group under complex multiplication. Yekutieli went one step further by proving that
$G(\Q)=U\times F$ where $U:=\{1,-1,i,-i\}$ and $F$ is a free group generated by\footnote{Note that such $x_0$ and $y_0$ always exist for every prime $p$ satisfying $p\equiv 1$(mod 4). In fact this choice of $x_0$ and $y_0$ is unique.}
$$\left\{\zeta_{p}:=\frac{x_0+y_0\sqrt{-1}}{x_0-y_0\sqrt{-1}}; \text{ where }p\equiv 1\ (\text{mod }4) \text{ and }x_0^2+y_0^2=p^2, x_0,y_0>0\right\}.$$
Thus we obtain a factorization of each of the elements of $G(\Q)$. Using this factorization he proved the main result of his paper.

\begin{thm}[Theorem 5.3 of \cite{Ye}]\label{thm:main_result_Ye}
    Let $c>1$ have prime factorization in $\Z$ as follows:
    $$c \ = \ p_1^{n_1}p_2^{n_2} \cdots p_k^{n_k}$$
    where the $p_i$'s are distinct primes and the $n_i$'s are positive integers.
    \begin{enumerate}
        \item If $p_i\equiv 1$ (mod 4) for all $i$, then the function \textbf{pt} restricts to a bijection
        $$pt:\{\zeta_{p_1}^{n_1}\zeta_{p_2}^{\pm n_2} \cdots \zeta_{p_k}^{\pm n_k}\}\to \textbf{PT}_c,$$
        where  $\textbf{PT}_c$ denotes the set of normalized solution of the form $(a,b,c)$. Thus there are $2^{k-1}$ distinct normalized solutions of the form $(a,b,c)$.
        \item If $p_i\not\equiv 1$ (mod 4) for some $i$ then $PT_c$ is empty i.e., there is no normalized solution of the form $(a,b,c)$.
    \end{enumerate}
\end{thm}

\subsection{Main results Jaklitsch et. al.}
In \cite{JMMM} we generalized the approach of Yekutieli to inspect the structure of integral solutions to 
\begin{equation}\label{eq:Pell Eq}
    x^2+Dy^2 \ = \ z^2
\end{equation}
for $D>1$. The goal is the same as that of \cite{Ye}, i.e., finding the number of normalized solutions to \eqref{eq:Pell Eq} of the form $(a,b,c)$ for a given positive integer $c$. However, we had to update the definition of normalized solutions for $D>1$.

\begin{defi}\label{defi:normalized solutions for D>1}
    A solution $(x_0,y_0,z_0)$ to \eqref{eq:Pell Eq} is defined to be normalized if $\gcd(x_0,y_0,z_0)=1$.
\end{defi}

Note that this definition is slightly different that Definition \ref{defi: normalized solutions for D=1} as in case of $D>1$ if we swap $(x_0,y_0,z_0)$ to $(y_0,x_0,z_0)$ then it will no longer be a solution to \eqref{eq:Pell Eq}. 

In order to achieve our goal, we generalized the arguments of Yekutieli and  focused on rational coordinate points on the ellipse $x^2+Dy^2=1$, which we can algebraically visualise as the following:
$$G_D(\Q) \ := \ \{a+b\sqrt{-D}: a,b\in \Q \text{ and }a^2+Db^2=1\}.$$
Note that this is also a group under the complex multiplication. One of the most important result of our paper is proving that $G_D(\Q)=U\times F$ where $U=\{1,-1\}$ and $F$ is a free group generated over the following set\footnote{Note the similarity of this result with the Mordell-Weil theorem for elliptic curves.}:
$$\left\{\zeta_{p}\ := \ \frac{x_0+y_0\sqrt{-D}}{p}; \text{ where } \left(\frac{-D}{p}\right)=1 \text{ and }x_0^2+Dy_0^2=p^2, x_0,y_0>0 \right\}$$
however we had to impose an \textbf{extra condition} of the class group $C(-4D)$ being a free $\Z_2$-module.

If we assume this extra condition, then due to Lemmas 3.3, 3.4 and 3.5 of \cite{JMMM}, such $\zeta_p$'s are well defined. We provide the statement of Lemma 3.5 of \cite{JMMM}, which uses Lemmas 3.3 and 3.4 of \cite{JMMM} and gives us that the $\zeta_p$'s are well defined.

\begin{lemma}[Lemma 3.5 of \cite{JMMM}]\label{lem: Lemma 3.5 of JMMM}
Let $p$ be an odd prime. Then the equation $x^2 + Dy^2 = z^2$ has a unique normalized solution of the form $(a,b,p)$.
\end{lemma}

Using Theorem 3.6 and Theorem 3.9 of \cite{JMMM}, we proved that $G_D(\Q)=U\times F$. We now state these important theorems.

\begin{thm}[Theorem 3.6 of \cite{JMMM}]\label{thm: Theorem 3.6 of JMMM}
    Let $z = \frac{a + b\sqrt{-D}}{c} \in G_D(\mathbb{Q})$ where $(a,b) = 1$, $c > 1$. Let $c = p_1^{\alpha_1} \cdots p_k^{\alpha_k}$. Then $$ z \ = \  \pm \zeta_{p_1}^{\pm \alpha_1}\cdots \zeta_{p_k}^{\pm \alpha_k}.$$
\end{thm}

\begin{thm}[Theorem 3.9 of \cite{JMMM}]\label{thm: Theorem 3.9 of JMMM}
    For some $z \in G_D(\mathbb{Q})$ let $z = \pm\zeta_{p_1}^{\pm \alpha_1} \cdots \zeta_{p_k}^{\pm \alpha_k}$ where $\leg{-D}{p_i} = 1$ for all $i$. Then if $z$ corresponds to the normalized triple $(a, b, c)$, we must have $c = p_1^{\alpha_1} \cdots p_k^{\alpha_k}$.
\end{thm}

Using these results we proved Theorem 4.1 of \cite{JMMM} which states the following and was the main result of that work.

\begin{thm}[Theorem 4.1 of \cite{JMMM}]\label{thm: Theorem 4.1 of JMMM}
    Given an integer $c>0$ with prime factorization $p_1^{n_1}\cdots p_k^{n_k}$ there are normalized solution of the form $(a,b,c)$ if and only if $\left(\frac{-D}{p_i}\right)=1$ for all $i\le k$. In fact if $\left(\frac{-D}{p_i}\right)=1$ for all $i\le k$ then there are exactly $2^{k-1}$ normalized solutions to \eqref{eq:Pell Eq} of the form $(a,b,c)$.
\end{thm}

Let $\Gamma$ be the abelian group of order 4 generated by multiplication by $-1$ and complex conjugation. In order to prove this theorem, we defined the sets
$$T_1\ :=\ \left\{\frac{a + b\sqrt{-D}}{c} : a^2 + Db^2 = c^2, (a,b) = 1, c > 1\right\},$$
and \be T_2 \ := \  \left\{\pm \zeta_{p_1}^{\epsilon_1 n_1} \zeta_{p_2}^{\epsilon_2 n_2} \cdots \zeta_{p_k}^{\epsilon_k n_k} : \epsilon_i \in \{\pm1\}\right\}.\ee Then we exhibited a bijection between $T_1$ and 
\be T_2/\Gamma \ = \  \left\{\zeta_{p_1}^{n_1} \zeta_{p_2}^{\epsilon_2 n_2} \cdots \zeta_{p_k}^{\epsilon_k n_k} : \epsilon_i \in \{\pm1\}\right\}\ee where $\Gamma$ is the abelian group of order 4 generated by complex conjugation and multiplication by $-1$.
Note that $T_2/\Gamma$ has $2^{k-1}$ many elements which explains why there are $2^{k-1}$ many normalized solutions.
However, proving these similar results for $D>1$ was more difficult than the $D=1$ case when $D=1$, Yekutieli heavily relied on the fact that $\Z[\sqrt{-1}]$ is a PID, which can fail for $\Z[\sqrt{-D}]$ where $D>1$. That is why we had to impose the extra condition of the class group $C(-4D)$ being a free $\Z_2$-module.

\ \\

In this survey article we fix a value of $D$, and verify the main results of \cite{JMMM} and compute everything explicitly to highlight the key ideas of the method. We choose $D=105$ so that the class group becomes $C(-420)=\Z_2\times \Z_2\times \Z_2$. We then choose certain values of $c$ and verify Theorem 3.6, Theorem 3.9, and particularly, Theorem 4.1 by explicitly finding the normalized solutions of the form $(a,b,c)$.

\section{Proof of the Main Results}

We provide examples to illustrate the theory we developed in \cite{JMMM}. Recall Pell's equation
\begin{equation}\label{eq: Pell eq}
    x^2+Dy^2=z^2.
\end{equation}
We are interested in positive integral solutions of this equation. Let us choose $D=105$. We state and prove the following claims.

\begin{cla}\label{cl: order less than 2}
    The class represented by the reduced form $ax^2+bxy+cy^2$ has order $\le 2$ in the group $C(-4D)$ \textit{if and only if} $b=0$, $a=b$ or $a=c$.
\end{cla}
\begin{proof}
Recall that a primitive positive definite form $[a,b,c]$ is said to be \textit{reduced} if $\abs{b}\le a\le c$ and $b\ge 0$ whenever $\abs{b}=a$ or $a=c$.
The proof of this claim is essentially Lemma 3.10 of \cite{Cox}.
\end{proof}
\begin{cla}
    The class group of binary quadratic form with discriminant $-4D=-420$ is isomorphic to $\Z_2\times \Z_2\times \Z_2$.
\end{cla}
\begin{proof}
In order to prove this claim, we explicitly find the  equivalence classes, i.e., elements of the class group $C(-420)$. 
The equivalence classes are
$$\overline{[1,0,105]}, \overline{[3,0,35]}, \overline{[5,0,21]}, \overline{[7,0,15]}, \overline{[2,2,53]}, \overline{[6,6,19]}, \overline{[11,8,11]}, \overline{[10,10,13]}$$
where $\overline{[a,b,c]}$ denotes the equivalence class represented by $[a,b,c]$. Note that each of these primitive forms are reduced and using Claim \ref{cl: order less than 2} we can infer that the all of the equivalence classes has order less than or equal to 2.
To compute these elements of $C(-420)$ we used a program which was based on the Algorithm 5.3.5 of \cite{Cohen}. The proof of this algorithm's correctedness is in fact Lemma 5.3.4 of the same book. So the cardinality of the class group $C(-420)$ is 8 and each element of this group has order $\leq 2$. Using the structure theorem of finite abelian groups, we infer that the class group $C(-420)$ is indeed isomorphic to $\Z_2\times \Z_2\times \Z_2$.
\end{proof}

We have that $C(-420)$ is a free $\Z_2$ module. Now, we choose five different values of $c$ and for each of those values we exhibit the results we developed in our paper. 

\begin{enumerate}
    \item $c=11\times 13$.
    \item $c=13\times 19$.
    \item $c=11\times 19$.
    \item $c=11^2\times 13^3$.
    \item $c=11\times 13\times 19$.
\end{enumerate}

Note that in each of the above choices of $c$, its prime factors are 11, 13 or 19, which satisfy $\left(\frac{-105}{11}\right)=\left(\frac{-105}{13}\right)=\left(\frac{-105}{19}\right)=1$. By Lemma 3.3 of the \cite{JMMM}, for each such choices of $c$, we will attain normalized solutions of the form $(a,b,c)$. Note that 
\begin{align*}
4^2+105\cdot 1^2&=(11)^2,\\
8^2+105\cdot 1^2&=(13)^2,\\
16^2+105\cdot 1^2&=(19)^2.
\end{align*}
Then we have
$$\zeta_{11}\ := \ \frac{4+\sqrt{-105}}{11},\  \ \ \zeta_{13}\ :=\ \frac{8+\sqrt{-105}}{13},\  \ \ \zeta_{19} \ := \ \frac{16+\sqrt{-105}}{19}.$$
Now, observe the following.

\begin{enumerate}
    \item By Theorem \ref{thm: Theorem 4.1 of JMMM}, we expect to have 2 normalized solutions. Since $11 \cdot 13$ has two prime factors, the theorem tells us we will have $2^1$ many normalized solutions. One can check that there are only two normalized solutions of the form $(a, b, 11 \times 13): (73, 12, 11\cdot 13)$ and $(137, 4, 11 \cdot 13)$. The solution $(73, 12, 11\cdot 13)$ corresponds to $\frac{73+12\sqrt{-105}}{11\cdot 13}\in G_D(\Q)$ and $(137, 4, 11\cdot 13)$ corresponds to $\frac{137+4\sqrt{-105}}{11\cdot 13}\in G_D(\Q)$ up to sign. Also, the following equalities hold:
    $$\frac{-73+12\sqrt{-105}}{11\cdot 13} \ = \ 
    \zeta_{11}\zeta_{13}$$
    
    $$\frac{137-4\sqrt{-105}}{11\cdot 13} \ = \
    \zeta_{11}^{-1}\zeta_{13}$$
    
    $$\frac{137+4\sqrt{-105}}{11\cdot 13} \ = \
    \zeta_{11}\zeta_{13}^{-1}$$
    
    $$\frac{-73-12\sqrt{-105}}{11\cdot 13} \ = \
    \zeta_{11}^{-1}\zeta_{13}^{-1}.$$
    
    Each of the normalized solutions $(a,b,11\cdot 13)$ corresponds to an element of the form $\pm \zeta_{11}^{\pm 1}\zeta_{13}^{\pm 1}$ in $G_D(\Q)$. However, since we are considering positive integral solutions only, there only 2 positive integral solutions to \eqref{eq: Pell eq} where $D=105$. This observation supports the claim of Theorem \ref{thm: Theorem 4.1 of JMMM}. In the case of this example, the bijection between $T_2/\Gamma$ and the set of normalized solutions is as follows:
    $$\overline{-\zeta_{11}^{-1}\zeta_{13}^{-1}} \ \mapsto\ (73,12,11 \cdot 13)$$
    $$\overline{\zeta_{11}\zeta_{13}^{-1}}\ \mapsto\
    (137, 4, 11\cdot 13),$$
    hence there are only 2 normalized solutions. We present the above information in the table below. 

    \par\vspace{0.2cm}
    \begin{center}
\begin{table}[ht]
    \begin{tabular}{| m{5cm} | m{5.5cm}|}
         \hline
  Expected number of solutions & $2^1 = 2$ \\ 
  \hline
  Actual solutions & $(73, 12, 11\cdot 13),  
  (137, 4, 11 \cdot 13)$\\ 
  \hline
  Solutions as elements of $G_D(\mathbb{Q})$ & $\frac{73+12\sqrt{-105}}{11\cdot 13}$, $\frac{137+4\sqrt{-105}}{11\cdot 13}$  \\ 
  \hline
  Factorization of elements &  $\frac{-73+12\sqrt{-105}}{11\cdot 13}=
    \zeta_{11}\zeta_{13}$
    
    $\frac{137-4\sqrt{-105}}{11\cdot 13}\ = \ 
    \zeta_{11}^{-1}\zeta_{13}$
    
    $\frac{137+4\sqrt{-105}}{11\cdot 13} \ = \
    \zeta_{11}\zeta_{13}^{-1}$
    
    $\frac{-73-12\sqrt{-105}}{11\cdot 13} \ = \
    \zeta_{11}^{-1}\zeta_{13}^{-1}$ \\ 
  \hline
  Bijection between $T_2/\Gamma$ and normalized solutions & $\overline{-\zeta_{11}^{-1}\zeta_{13}^{-1}}\ \mapsto \ (73,12,11 \cdot 13)$
    $\overline{\zeta_{11}\zeta_{13}^{-1}} \ \mapsto\ 
    (137, 4, 11\cdot 13)$  \\ 
  \hline
\end{tabular}
\caption{$D = 11\cdot 13$.}
\end{table}
\end{center}
    
    \item Again we have two prime factors, so \ref{thm: Theorem 4.1 of JMMM} states that there should be 2 normalized solutions. The two normalized solutions of the form $(a, b, 13 \times 19)$ are $(23, 24, 13\cdot 19)$ and $(233, 8, 13 \cdot 19)$. Writing these solutions as elements of $G_D(\Q)$, we get $\frac{23+24\sqrt{-105}}{13\cdot 19}$ and $\frac{233+8\sqrt{-105}}{13\cdot 19}$. We can then represent the elements of $G_D(\Q)$ as products of $\zeta_p$ terms: 
    $$\frac{23+24\sqrt{-105}}{13\cdot 19}\ = \
    \zeta_{13}\zeta_{19}$$
    
    $$\frac{233-8\sqrt{-105}}{13\cdot 19}\ =\ 
    \zeta_{13}^{-1}\zeta_{19}$$
    
    $$\frac{233 + 8\sqrt{-105}}{13\cdot 19} \ = \
    \zeta_{13}\zeta_{19}^{-1}$$
    
    $$\frac{23-24\sqrt{-105}}{13\cdot 19}\ = \ 
    \zeta_{13}^{-1}\zeta_{19}^{-1}.$$
    
    Therefore, each solution of \eqref{eq: Pell eq} corresponds to an element of the form $\pm \zeta_{13}^{\pm 1}\zeta_{19}^{\pm 1}$ in $G_D(\Q)$. We only consider positive integral solutions, which correspond to the elements $\zeta_{13}\zeta_{19}$. Therefore, the bijection between $T_2/\Gamma$ and the set of normalized solutions is given by
    $$\overline{\zeta_{13}\zeta_{19}}\ \mapsto\ (23,24,13 \cdot 19)$$
    $$\overline{\zeta_{13}\zeta_{19}^{-1}}\ \mapsto\ 
    (233, 8, 13\cdot 19).$$
    \par\vspace{0.2cm}
    
\begin{center}
\begin{table}[ht]
\begin{tabular}{| m{5cm} | m{5.5cm}|}
         \hline
  Expected number of solutions & $2^1 = 2$ \\ 
  \hline
  Actual solutions & $(23, 24, 13\cdot 19),  
  (233, 8, 13 \cdot 19)$\\ 
  \hline
  Solutions as elements of $G_D(\mathbb{Q})$ & $\frac{23+24\sqrt{-105}}{13\cdot 19}$, $\frac{233+8\sqrt{-105}}{13\cdot 19}$  \\ 
  \hline
  Factorization of elements &  $\frac{23+24\sqrt{-105}}{13\cdot 19}\ = \ 
    \zeta_{13}\zeta_{19}$
    
    $\frac{233-8\sqrt{-105}}{13\cdot 19} \ = \ 
    \zeta_{13}^{-1}\zeta_{19}$
    
    $\frac{233+8\sqrt{-105}}{13\cdot 19}\ = \
    \zeta_{13}\zeta_{19}^{-1}$
    
    $\frac{23-24\sqrt{-105}}{13\cdot 19}=
    \zeta_{13}^{-1}\zeta_{19}^{-1}$ \\ & \\
  \hline
  Bijection between $T_2/\Gamma$ and normalized solutions & $\overline{\zeta_{13}\zeta_{19}}\mapsto (23,24,13 \cdot 19)$
    $\overline{\zeta_{13}\zeta_{19}^{-1}}\mapsto
    (233, 8, 13\cdot 19)$  \\ 
  \hline
\end{tabular}
\caption{$D=13\cdot 19$.}
\end{table}
\end{center}
    
    \item This example follows a very similar format to the two preceding it. We get 2 normalized solutions of the form $(a, b, 11 \times 19): (41, 20, 11\cdot 19)$ and $(169, 12, 11 \cdot 19)$, which is what  \ref{thm: Theorem 4.1 of JMMM} predicts. If we write these as solutions as elements of $G_D(\Q)$ we get $\frac{41+20\sqrt{-105}}{11\cdot 19}$ $\frac{169+12\sqrt{-105}}{11\cdot 19}$. We write the factorization of solutions of the form $(a, b, 11 \times 19)$:
    $$\frac{-41+20\sqrt{-105}}{11\cdot 19}\ = \ 
    \zeta_{11}\zeta_{19}$$
    
    $$\frac{169-12\sqrt{-105}}{11\cdot 19} \ = \ 
    \zeta_{11}^{-1}\zeta_{19}$$
    
    $$\frac{169+12\sqrt{-105}}{11\cdot 19} \ = \ 
    \zeta_{11}\zeta_{19}^{-1}$$
    
    $$\frac{-41-20\sqrt{-105}}{11\cdot 19}\ = \ 
    \zeta_{11}^{-1}\zeta_{19}^{-1}.$$
    
    The positive solutions then correspond to the elements $\zeta_{13}\zeta_{19}$ and $\zeta_{13}\zeta_{19}^{-1}$, so the bijection between $T_2/\Gamma$ and the set of normalized solutions is given by
    $$\overline{-\zeta_{11}^{-1}\zeta_{19}^{-1}}\ \mapsto\  (41,20,11 \cdot 19)$$
    $$\overline{\zeta_{11}\zeta_{19}^{-1}}\ \mapsto\ 
    (169, 12, 11\cdot 19).$$
    
\begin{center}
\begin{table}[H]
\begin{tabular}{ | m{5cm} | m{5.5cm}| } 
\hline
  Expected number of solutions & $2^1 = 2$ \\ 
  \hline
  Actual solutions & $(41, 20, 11\cdot 19),  
  (169, 12, 11 \cdot 19)$\\ 
  \hline
  Solutions as elements of $G_D(\mathbb{Q})$ & $\frac{41+20\sqrt{-105}}{11\cdot 19}$, $\frac{169+12\sqrt{-105}}{11\cdot 19}$  \\ 
  \hline
  Factorization of elements &  $\frac{-41+20\sqrt{-105}}{11\cdot 19}\ = \ 
    \zeta_{11}\zeta_{19}$
    
    $\frac{169-12\sqrt{-105}}{11\cdot 19}\ = \ 
    \zeta_{11}^{-1}\zeta_{19}$
    
    $\frac{169+12\sqrt{-105}}{11\cdot 19} \ = \ 
    \zeta_{11}\zeta_{19}^{-1}$
    
    $\frac{-41-20\sqrt{-105}}{11\cdot 19}\ =\ 
    \zeta_{11}^{-1}\zeta_{19}^{-1}$ \\ & \\
  \hline
  Bijection between $T_2/\Gamma$ and normalized solutions & $\overline{-\zeta_{11}^{-1}\zeta_{19}^{-1}} \ \mapsto\  (41,20,11 \cdot 19)$
    $\overline{\zeta_{11}\zeta_{19}^{-1}} \ \mapsto\ 
    (169, 12, 11\cdot 19)$  \\ 
  \hline
\end{tabular}
\caption{$D=11\cdot 19$.}
\end{table}
\end{center}
    
    \item In this example, the number $11^2 \cdot 13^2$ still only has 2 distinct prime factors, so \ref{thm: Theorem 4.1 of JMMM} states that there will 2 normalized solutions of the form $(a,b,11^2\times 13^3)$. One can check that these solutions are $(251792, 8321, 11^2\cdot 13^3)$ and $(105632, 23807, 11^2\cdot 13^3)$. Writing these solutions as elements of $G_D(\Q)$ we get $\frac{251792+8321\sqrt{-105}}{11^2\cdot 13^3}$ and $\frac{105632+23807\sqrt{-105}}{11^2\cdot 13^3}$. We factor these elements into products of $\zeta_p$ by:
    $$\frac{251792+8321\sqrt{-105}}{11^2\cdot 13^3}\ = \
    \left(\zeta_{11}\right)^{-2}\left(\zeta_{13}\right)^{3}$$
    
    $$\frac{251792-8321\sqrt{-105}}{11^2\cdot 13^3} \ = \ 
    \left(\zeta_{11}\right)^{2}\left(\zeta_{13}\right)^{-3}$$
    
    $$\frac{105632-23807\sqrt{-105}}{11^2\cdot 13^3} \ = \ 
    \left(\zeta_{11}\right)^{2}\left(\zeta_{13}\right)^{3}$$
    
    $$\frac{105632+23807\sqrt{-105}}{11^2\cdot 13^3} \ = \ 
    \left(\zeta_{11}\right)^{-2}\left(\zeta_{13}\right)^{-3}.$$
    
    The positive normalized solutions are $\frac{251792+8321\sqrt{-105}}{11^2\cdot 13^3}$ and $\frac{105632+23807\sqrt{-105}}{11^2\cdot 13^3}$, which correspond to the elements $\left(\zeta_{11}\right)^{-2}\left(\zeta_{13}\right)^{3}$ and $\left(\zeta_{11}\right)^{-2}\left(\zeta_{13}\right)^{-3}$ respectively. Therefore, the bijection between $T_2/\Gamma$ and the set of normalized solutions is as follows:
    $$\overline{\zeta_{11})^{-2}(\zeta_{13})^3} \ \mapsto\  (251792,8321,11^213^3)$$
    $$\overline{(\zeta_{11})^{-2}(\zeta_{13})^{-3}}\ \mapsto \ 
    (105632, 23807, 11^213^3).$$
    
\begin{center}
\begin{table}[H]
\begin{tabular}{ | m{5cm} | m{7cm}| } 
\hline
  Expected number of solutions & $2^1 = 2$ \\ 
  \hline
  Actual solutions & $(251792, 8321, 11^2\cdot13^3)$,   
  \\ & $(105632, 23807, 11^2 \cdot 13^3)$\\ 
  \hline
  Solutions as elements of $G_D(\mathbb{Q})$ & $\frac{251792+8321\sqrt{-105}}{11^2\cdot 13^3}$, $\frac{105632+23807\sqrt{-105}}{11^2\cdot 13^3}$  \\ 
  \hline
  Factorization of elements &  $\frac{105632- 23807\sqrt{-105}}{11^2\cdot 13^3} \ = \ 
    \zeta_{11}^2\zeta_{13}^3$

    $\frac{251792+ 8321\sqrt{-105}}{11^2\cdot 13^3} \ = \ 
    \zeta_{11}^{-2}\zeta_{13}^3$
    
    $\frac{251792-8321\sqrt{-105}}{11^2\cdot 13^3} \ = \ 
    \zeta_{11}^2\zeta_{13}^{-3}$
    
    $\frac{105632+23807\sqrt{-105}}{11^2\cdot 13^3} \ = \ 
    \zeta_{11}^{-2}\zeta_{13}^{-3}$ \\ & \\ 
  \hline
  Bijection between $T_2/\Gamma$ and normalized solutions &  $\overline{-\zeta_{11}^{-2}\zeta_{13}^3}\ \mapsto \ (251792,8321,11^2 \cdot 13^3)$ \\ & \\ &
    $\overline{\zeta_{11}^{-2}\zeta_{13}^{-3}}\ \mapsto\
    (105632, 23807, 11^2\cdot 13^3)$  \\ 
    
  \hline
\end{tabular}
\caption{$D=11^2\cdot 13^3$.}
\end{table}
\end{center}
\vspace{0.5cm}

    \item This example differs from the previous ones, because there are 3 prime factors of $11\cdot 13 \cdot 19$ rather than 2. Therefore, Theorem \ref{thm: Theorem 4.1 of JMMM} states that there will be $2^2 = 4$ normalized solutions. One can check that these solutions are 
    $(2612, 73, 11\cdot 13\cdot 19),
    (2428, 119, 11\cdot 13\cdot 19),
    (1772, 201, 11\cdot 13\cdot 19),$ and 
    $(92, 265, 11\cdot 13\cdot 19)$.
    We factor these solutions, written as elements of $G_D(\Q)$ as follows: 
    $$\frac{-2428+119\sqrt{-105}}{11\cdot 13\cdot 19}=\zeta_{11}\zeta_{13}\zeta_{19},\ \frac{2612+73\sqrt{-105}}{11\cdot 13\cdot 19}=\zeta_{11}^{-1}\zeta_{13}\zeta_{19}$$
    $$\frac{1772+201\sqrt{-105}}{11\cdot 13\cdot 19}=\zeta_{11}\zeta_{13}^{-1}\zeta_{19}, \ \frac{92-265\sqrt{-105}}{11\cdot 13\cdot 19}=\zeta_{11}^{-1}\zeta_{13}^{-1}\zeta_{19}$$
    $$\frac{92+265\sqrt{-105}}{11\cdot 13\cdot 19}=\zeta_{11}\zeta_{13}\zeta_{19}^{-1},\ \frac{1772-201\sqrt{-105}}{11\cdot 13\cdot 19}=\zeta_{11}^{-1}\zeta_{13}\zeta_{19}^{-1}$$
    $$\frac{2612-73\sqrt{-105}}{11\cdot 13\cdot 19}=\zeta_{11}\zeta_{13}^{-1}\zeta_{19}^{-1}, \ \frac{-2428-119\sqrt{-105}}{11\cdot 13\cdot 19}=\zeta_{11}^{-1}\zeta_{13}^{-1}\zeta_{19}^{-1}.$$
    
    So, each of the normalized solutions of the form $(a,b,11\cdot 13\cdot 19)$ corresponds to the elements of the form $\pm \zeta_{11}^{\pm 1}\zeta_{13}^{\pm 1}\zeta_{19}^{\pm 1}$ in $G_D(\Q)$. The positive solutions are given by $\zeta_{11}^{-1}\zeta_{13}\zeta_{19}, -\zeta_{11}^{-1}\zeta_{13}^{-1}\zeta_{19}^{-1}, \zeta_{11}\zeta_{13}^{-1}\zeta_{19},$ and $\zeta_{11}\zeta_{13}\zeta_{19}^{-1}$. Therefore, the bijection between $T_2/\Gamma$ and the set of normalized solutions is as follows:
    $$\overline{\zeta_{11}^{-1}\zeta_{13}\zeta_{19}}\mapsto (2612, 73, 11\cdot 13 \cdot 19)$$
    $$\overline{-\zeta_{11}^{-1}\zeta_{13}^{-1}\zeta_{19}^{-1}}\mapsto
    (2428, 119, 11\cdot 13 \cdot 19)$$
    $$\overline{\zeta_{11}\zeta_{13}^{-1}\zeta_{19}}\mapsto (1772, 201, 11\cdot 13 \cdot 19)$$
    $$\overline{\zeta_{11}\zeta_{13}\zeta_{19}^{-1}}\mapsto
    (92, 265, 11\cdot 13 \cdot 19).$$

\begin{center}
\begin{table}[H]
\begin{tabular}{ | m{5cm} | m{7 cm}| } 
\hline
  Expected number of solutions & $2^2 = 4$ \\ 
  \hline
  Actual solutions & $(2612, 73, 11\cdot 13\cdot 19)$,  
  \\ & $(2428, 119, 11 \cdot 13\cdot 19)$, 
  \\ & $(1772, 201, 11 \cdot 13\cdot 19)$, 
  \\ & $(92, 265, 11 \cdot 13\cdot 19)$\\ 
  \hline
  Solutions as elements of $G_D(\mathbb{Q})$ & $\frac{2612+73\sqrt{-105}}{11\cdot 13\cdot 19}$, $\frac{2428+119\sqrt{-105}}{11\cdot 13\cdot 19}$,
  \\ & $\frac{1772+201\sqrt{-105}}{11\cdot 13\cdot 19}$, $\frac{92+265\sqrt{-105}}{11\cdot 13\cdot 19}$ \\ &
  { } \\
  \hline
  Factorization of elements &  $\frac{-2428+119\sqrt{-105}}{11\cdot 13\cdot 19}\ = \ 
    \zeta_{11}\zeta_{13}\zeta_{19}$
    
    $\frac{2612+73\sqrt{-105}}{11\cdot 13\cdot 19}\ = \ 
    \zeta_{11}^{-1}\zeta_{13}\zeta_{19}$
    
    $\frac{1772+201\sqrt{-105}}{11\cdot 13\cdot 19}\ = \
    \zeta_{11}\zeta_{13}^{-1}\zeta_{19}$
    
    $\frac{92-265\sqrt{-105}}{11\cdot 13\cdot 19}\ = \ 
    \zeta_{11}^{-1}\zeta_{13}^{-1}\zeta_{19}$
    
    $\frac{92+ 265\sqrt{-105}}{11\cdot 13\cdot 19} \ = \ 
    \zeta_{11}\zeta_{13}\zeta_{19}^{-1}$
    
    $\frac{1772-201\sqrt{-105}}{11\cdot 13\cdot 19}\ = \ 
    \zeta_{11}^{-1}\zeta_{13}\zeta_{19}^{-1}$
    
    $\frac{2612-73\sqrt{-105}}{11\cdot 13\cdot 19}\ = \ 
    \zeta_{11}\zeta_{13}^{-1}\zeta_{19}^{-1}$
    
    $\frac{-2428-119\sqrt{-105}}{11\cdot 13\cdot 19} \ = \ 
    \zeta_{11}^{-1}\zeta_{13}^{-1}\zeta_{19}^{-1}$\\ & \\ 
  \hline
  Bijection between $T_2/\Gamma$ and normalized solutions & $\overline{\zeta_{11}^{-1}\zeta_{13}\zeta_{19}}\mapsto (2612,73,11 \cdot 13\cdot 19)$
    $\overline{-\zeta_{11}^{-1}\zeta_{13}^{-1}\zeta_{19}^{-1}} \mapsto
    (2428, 119, 11\cdot 13\cdot 19)$ 
    $\overline{\zeta_{11}\zeta_{13}^{-1}\zeta_{19}}\mapsto (1772,201,11 \cdot 13\cdot 19)$
    $\overline{\zeta_{11}\zeta_{13}\zeta_{19}^{-1}}\mapsto
    (92, 265, 11\cdot 13\cdot 19)$\\ 
  \hline
\end{tabular}
\caption{$D=11\cdot13\cdot 19$.}
\end{table}
\end{center}

\end{enumerate}

From the above set of examples we observe the following.

\begin{itemize}
    \item From examples (1), (2), (3) we note that the prime factors (other than satisfying $\left(\frac{-420}{p}=1\right)$ where $p=11,13,19$) in particular has no contribution to the number of normalized solutions. In each of these examples we have changed the set of prime factors yet the number of distinct normalized solution of the form $(a,b,c)$ is always 2.
    \item From (4) we observe that the exponent of the prime factors of $c$ has no significance in case of determining the number of distinct normalized solutions. In this particular examples we have raised $11$ and $13$ to higher exponents yet the number of normalized solutions of the form $(a,b,c)$ is still 2.
    \item From $(5)$ we note that the number of prime factors takes importance in case of determining the number of distinct normalized solutions. Here the number of primes factors of $c$ is no longer 2 unlike the earlier cases. Consequently we note that the normalized solution of the form $(a,b,c)$ is now $4$ instead of $2$.
\end{itemize}      

From these above set of examples we have verified the main results of \cite{JMMM}, i.e., Theorem 3.6, Theorem 3.9 and Theorem 4.1.

The theorem holds for values of $D$ when when $-D \equiv 2,3$ mod 4 and when $C(-4D)$ is a free $\mathbb{Z}_2$ module. There are 65 numbers called Euler's convenient numbers in \cite{Cox}. They are numbers $n$ such that $C(-4n)$ is a free $\mathbb{Z}_2$ module. They are conjectured to be the only numbers $n$ with this property, but there can be at most two more such numbers. Therefore if we take the subset of these numbers that are square free and equivalent to 1 or 2 mod 4 we get the following list, these are the known values of $D$ for which the theorem holds: 1, 2, 5, 6, 8, 10, 13, 21, 22, 30, 33, 37, 42, 57, 58, 70, 78, 85, 93, 102, 105, 130, 133, 165, 177, 190, 
210, 253, 273, 330, 345, 357, 385, 462, and 1365.

Our analysis crucially uses the structure of the class group, in particular that each element has order two. A natural next project is to investigate what happens for other structures, and see if similarly comprehensive classifications hold.


\ \\

\end{document}